\documentclass[12pt,table]{amsart}

\usepackage{amsmath}
\usepackage[all]{xy}
\usepackage[monochrome]{xcolor}
\usepackage{amssymb,tikz}
\usepackage{slashbox,booktabs,amsmath}
\usepackage{graphicx}
\usepackage{adjustbox}
\usepackage{tikz}
\usepackage{tikz-cd}
\usetikzlibrary{cd}

\usepackage{rotating}

\def\BibTeX{{\rm B\kern-.05em{\sc i\kern-.025em b}\kern-.08em
    T\kern-.1667em\lower.7ex\hbox{E}\kern-.125emX}}

\newtheorem{thm}{Theorem}[section]
\newtheorem{cor}[thm]{Corollary}
\newtheorem{lem}[thm]{Lemma}
\newtheorem{prop}[thm]{Proposition}
\theoremstyle{definition}
\newtheorem{exmp}[thm]{Example}
\theoremstyle{remark}
\newtheorem{rem}[thm]{Remark}
\theoremstyle{definition}
\newtheorem{defn}[thm]{Definition}

\numberwithin{equation}{section}

\begin{document}

\title[Generalized Stability of Heisenberg Coefficients] 
{Generalized Stability of Heisenberg coefficients}

%
\author{Li Ying}
\address{Li Ying, Department of Mathematics, Texas A\&M University, College Station, Texas, 77843, USA}
\email[]{98yingli@math.tamu.edu}
\urladdr{http://www.math.tamu.edu/~98yingli}
\thanks{}

\date{}


\begin{abstract}
Stembridge introduced the notion of stability for Kronecker triples which generalize Murnaghan's classical stability result for Kronecker coefficients. Sam and Snowden proved a conjecture of Stembridge concerning stable Kronecker triple, and they also showed an analogous result for Littlewood--Richardson coefficients. Heisenberg coefficients are Schur structure constants of the Heisenberg product which generalize both Littlewood--Richardson coefficients and Kronecker coefficients. We show that any stable triple for Kronecker coefficients or Littlewood--Richardson coefficients also stabilizes Heisenberg coefficients, and we classify the triples stabilizing Heisenberg coefficients. We also follow Vallejo's idea of using matrix additivity to generate Heisenberg stable triples.
\end{abstract}


\maketitle

\section{Introduction}
\label{Intro}
We assume familiarity with the basic results in the (complex) representation theory of symmetric groups and symmetric functions (see \cite{Mac, sagan}). There are two famous structure constants, Kronecker coefficients and Littlewood--Richardson coefficients, which can be defined in terms of representation of symmetric groups. Given a partition $\lambda$ of $n$ (written as $\lambda \vdash n$, or $\lambda$ has size $n$), let $V_\lambda$ be the associated irreducible representation of the symmetric group $S_n$. The \textit{Kronecker coefficient} $g_{\mu, \nu}^{\lambda}$ is the multiplicity of $V_\lambda$ in the irreducible decomposition of $\text{Res}^{S_n\times S_n}_{S_n} (V_\mu \otimes V_\nu)$, the \textit{Kronecker product} of $V_\mu$ and $V_\nu$. That is,
$$g_{\mu, \nu}^{\lambda}=\langle\,\text{Res}^{S_n\times S_n}_{S_n} (V_\mu \otimes V_\nu)\,,\, V_\lambda\,\rangle,$$
where $\lambda$, $\mu$, and $\nu$ are partitions of $n$, and $\langle\;,\; \rangle$ denotes the Hall inner product. The \textit{Littlewood--Richardson coefficient} $c_{\mu, \nu}^{\lambda}$ is the multiplicity of $V_\lambda$ in the irreducible decomposition of $\text{Ind}_{S_n\times S_m}^{S_{n+m}} (V_\mu \otimes V_\nu)$, the \textit{induction product} of $V_\mu$ and $V_\nu$. That is,
$$c_{\mu, \nu}^{\lambda}=\langle\,\text{Ind}_{S_n\times S_m}^{S_{n+m}} (V_\mu \otimes V_\nu)\,,\, V_\lambda\,\rangle,$$
where $\lambda\vdash n+m$, $\mu\vdash n$, and $\nu\vdash m$ for some positive integers $n$ and $m$. 

We view partitions as vectors so we define addition, subtraction, and scalar multiplication on them. While the Littlewood--Richardson coefficient is well-studied and has several beautiful combinatorial interpretations (see \cite{KnutsonTao, Mac, CombinLR}), an explicit combinatorial or geometric description for the Kronecker coefficient is still unknown. In 1938 Murnaghan \cite{F-1938} discovered a remarkable stability property for the Kronecker coefficients. He stated without proof that for any partitions $\lambda$, $\mu$, and $\nu$ of the same size, the sequence $\left\{g_{\mu+(n), \nu+(n)}^{\lambda+(n)}\right\}$ is eventually constant. There are many proofs with different flavours for this fact, see \cite{E-2010, LWStab, JT}. Stembridge \cite{GenStab} vastly generalized this result by introducing the concept of a stable triple.

\begin{defn}\label{K-stableDef}
	A triple $(\alpha, \beta, \gamma)$ of partitions of the same size with $g_{\beta, \gamma}^\alpha>0$ is a \textit{K-triple}. It is \textit{K-stable} if, for any other triple of partitions $(\lambda, \mu, \nu)$ with $|\lambda|=|\mu|=|\nu|$, the sequence $\left\{g_{\mu+n\beta, \nu+n\gamma}^{\lambda+n\alpha}\right\}$ is eventually constant.
\end{defn}
Thus, Murnaghan showed that $((1),(1),(1))$ is K-stable. Stembridge conjectured a characterization for K-stability, and he proved its necessity. Sam and Snowden \cite{SamSnowden} proved the sufficiency.

\begin{prop}\label{K-stable}
	A K-triple $(\alpha, \beta, \gamma)$ is K-stable if and only if $g_{n\beta, n\gamma}^{n\alpha}=1$ for all $n>0$.
\end{prop}

Sam and Snowden \cite{SamSnowden} also proved an analogous result for Littlewood--Richardson coefficients, which can also be deduced from some earlier work (see \cite[Remark 4.7]{SamSnowden}).

\begin{defn}
	A triple $(\alpha, \beta, \gamma)$ of partitions with $|\alpha|=|\beta|+|\gamma|$ and $c_{\beta, \gamma}^\alpha>0$ is an \textit{LR-triple}. It is \textit{LR-stable} if, for any other triple of partitions $(\lambda, \mu, \nu)$ with $|\lambda|=|\mu|+|\nu|$, the sequence $\left\{c_{\mu+n\beta, \nu+n\gamma}^{\lambda+n\alpha}\right\}$ is eventually constant.
\end{defn}
 
\begin{prop}[\cite{SamSnowden} Theorem 4.6]\label{LR-StabResult}
	The following are equivalent for an LR-triple $(\alpha, \beta, \gamma)$, 
	\begin{itemize}
		\item[(a)] $(\alpha, \beta, \gamma)$ is LR-stable.
		\item[(b)] $c_{\beta, \gamma}^{\alpha}=1.$
		\item[(c)] $c_{n\beta, n\gamma}^{n\alpha}=1$ for all $n>0.$
	\end{itemize}
\end{prop}
\begin{rem}
	Sam and Snowden \cite{SamSnowden} did not require $c_{\beta, \gamma}^\alpha>0$, which should be added. For example, when $\beta$ is not contained in $\alpha$, we have that $c_{\beta, \gamma}^\alpha=0$ and $\left\{c_{\mu+n\beta, \nu+n\gamma}^{\lambda+n\alpha}\right\}$ is eventually zero.
\end{rem}
Aguiar, Ferrer Santos, and Moreira introduced a new (commutative and associative) product, the Heisenberg product (denoted by $\#$), on representations of symmetric groups in \cite{M-2015} and \cite{Moreira}. This product interpolates between the induction product and the Kronecker product (see Definition \ref{DefHeisen} for details), its structure constants are Heisenberg coefficients, which are a common generalization of the Littlewood--Richardson coefficient and the Kronecker coefficient. We show that K-stable triples and LR-stable triples also stabilize Heisenberg coefficients, and we characterize the triples which do this.

Manivel \cite{ManKronCone} and Vallejo \cite{VallejoStab} independently generated K-stable triples using additive matrices. Manivel also showed that the set of stable triples is the intersection of the Kronecker semigroup with a union of faces of the Kronecker cone \cite[Proposition 2]{ManKronCone2}. Later, Pelletier \cite{Pelletier1} produced particular faces of the Kronecker cone containing only stable triples, which generalized Manivel and Vallejo's work. We will use the idea of additive matrices to generate triples of partitions which stabilize the Heisenberg coefficients.

This paper is organized as follows. In the second section, we will give the definition of the Heisenberg product and some related results, and define H-stable triple for Heisenberg coefficients. Section \ref{StabHeisenberg} shows that the K-stable triples and LR-stable triples are H-stable, and gives a necessary and sufficient condition for a triple to be H-stable. In Section \ref{SecAdditive}, we define H-additive matrices which generalize additive matrices, and show some results for H-additive matrices which are analogous to those for additive matrices. In the last Section, we prove that each H-additive matrix gives an H-stable triple.

\section{Heisenberg Product}
\label{Heisenberg Product}
\begin{defn}\label{DefHeisen}
	Let $V$ and $W$ be representations of $S_n$ and $S_m$ respectively. Fix an integer $l$ (weakly) between $\text{max}\{m,n\}$ and $m+n$, and let $p=l-m$, $q=n+m-l$, and $r=l-n$. The following diagram of inclusions (solid arrows) commutes:	
	\begin{equation}
	\label{IncluDiag}
	\begin{tikzcd}
	S_p\times S_q\times S_q \times S_r \arrow[r, hook] & {S_{p+q} \times S_{q+r} =S_n\times S_m} \arrow[ddl, bend right=16, dashed, "\text{Res}"'] & \hskip -6mm \color{red}{V\otimes W}\\
	&&\\
	{S_p \times S_q \times S_r} \arrow[r, hook] \arrow[uu, hook, "id_{S_p}\times \Delta_{S_q}\times id_{S_r}"] \arrow[uur, hook] \arrow[r, bend left=16, dashed, "\text{Ind}"]& {S_{p+q+r} =S_l} & \hskip -25mm \color{red}{(V \# W)_l}
	\end{tikzcd}
	\end{equation}
	The \textit{Heisenberg product} (denoted by $\#$) of $V$ and $W$ is
	\begin{equation}
	\label{HeisenbergProduct}
	\begin{split}
	V \# W= \bigoplus\limits_{l=\text{max}\{n,m\}}^{n+m} (V\# W)_l,
	\end{split}
	\end{equation}
	where the degree $l$ component is defined using the dashed arrows in \eqref{IncluDiag}:
	\begin{equation}
	\label{DegreeComponent}
	(V\# W)_l=\text{Ind}_{S_p\times S_q\times S_r}^{S_l}\text{Res}_{S_p\times S_q\times S_r}^{S_n\times S_m} (V\otimes W).
	\end{equation}
\end{defn}
\noindent The Heisenberg product connects the induction product and the Kronecker product. When $l= m+ n$, $(V\# W)_l$ is the induction product $\text{Ind}_{S_n\times S_m}^{S_{n+m}} (V\otimes W)$. When $l=n=m$, $(V\# W)_l$ is the Kronecker product $\text{Res}_{S_l}^{S_l\times S_l} (V\otimes W)$. Remarkably, this product is associative \cite[Theorem 2.3, Theorem 2.4, Theorem 2.6]{M-2015}. The \textit{Heisenberg coefficient} $h_{\mu, \nu}^\lambda$ is the multiplicity of $V_\lambda$ in $V_\mu\# V_\nu$,
$$h_{\mu, \nu}^\lambda=\langle\, V_\mu\# V_\nu\,,\, V_\lambda\,\rangle.$$
The Heisenberg coefficient generalizes the Littlewood-Richardson coefficient and the Kronecker coefficient. In \cite{LY}, we gave a formula for Heisenberg coefficients.
\begin{prop}
	\label{formulaheisenberg}
	Given partitions $\lambda\vdash l$, $\mu\vdash m$, and $\nu\vdash n$,
	\begin{equation}
	\label{h=cg}
	h_{\mu, \nu}^\lambda=\sum\limits_{\shortstack{\scriptsize {$\alpha \vdash p, \rho \vdash r$, \scriptsize $\tau \vdash n$}\\ \scriptsize {$\beta, \eta, \delta \vdash q$}}} c_{\alpha, \beta}^\mu\,\, c_{\eta, \rho}^\nu\,\, g_{\beta, \eta}^\delta\,\, c_{\alpha, \delta}^\tau\,\, c_{\tau, \rho}^\lambda
	\end{equation}
	where $\max\{m,n\}\leq l\leq m+n$, $p=l-n$, $q=m+n-l$, and $r=l-m$.
\end{prop}

In \cite{LY}, we showed that Heisenberg coefficients stabilize in low degrees, which is analogous to Murnaghan's stability result. It is worthwhile trying to also generalize stability for Heisenberg coefficients.

\begin{defn}
		A triple $(\alpha, \beta, \gamma)$ of partitions with $\max\{|\beta|, |\gamma|\}\leq|\alpha|\leq|\beta|+|\gamma|$ and $h_{\beta, \gamma}^\alpha>0$ is an \textit{H-triple}. It is \textit{H-stable} if, for any other triple of partitions $(\lambda, \mu, \nu)$ with $\max\{|\mu|, |\nu|\}\leq|\lambda|\leq|\mu|+|\nu|$, the sequence $\left\{h_{\mu+n\beta, \nu+n\gamma}^{\lambda+n\alpha}\right\}$ is eventually constant.
\end{defn}
\noindent Our result in \cite{LY} is that $((1), (1), (1))$ is an H-stable triple.

\section{H-stable triples}
\label{StabHeisenberg}
We show that the K-stable triples and LR-stable triples are H-stable. As in \cite{LY}, we begin with a stability result for Littlewood--Richardson coefficients.

\begin{lem}\label{LRLamma}
	Given partitions $\lambda$, $\mu$, $\nu$, $\alpha$, and a positive integer $n\geq |\mu|$ with $|\lambda+n\alpha|=|\mu|+|\nu|$ and $\nu\subset \lambda+n\alpha$, then
	\begin{itemize}
		\item[(1)] $\nu-(n-|\mu|)\alpha$ is a partition.
		
		\item[(2)] When $n$ is large, we have $c_{\mu, \nu}^{\lambda+n\alpha}=c_{\mu, \nu+\alpha}^{\lambda+(n+1)\alpha}.$
	\end{itemize}
\end{lem}

\begin{proof}
For part (1), we first show that $(n-|\mu|)\alpha\subset \nu$. It is enough to prove that $((n-|\mu|)\alpha)_i\leq \nu_i$ for all $i$. When $\mu=(0)$ (empty partition) or $\alpha_i=0$, this is trivially true. We consider the nontrivial case. Since
$$|\mu|=|\lambda+n\alpha|-|\nu|\geq (\lambda+n\alpha)_i-\nu_i,$$
we have 
$$\nu_i\geq (\lambda+n\alpha)_i-|\mu|\geq n\alpha_i-|u|=(n-|\mu|)\alpha_i+|\mu|(\alpha_i-1)\geq ((n-|\mu|)\alpha)_i.$$
We then show that $\nu-(n-|\mu|)\alpha$ is a partition. To see this, it suffices to show that
$$(\nu-(n-|\mu|)\alpha)_i\geq (\nu-(n-|\mu|)\alpha)_{i+1}, \hskip 3mm \text{for all} \;\;i,$$
that is,
\begin{equation}\label{-partition}
	\nu_i-\nu_{i+1}\geq(n-|\mu|)(\alpha_i-\alpha_{i+1}).
\end{equation}
This is obviously true when $\alpha_i=\alpha_{i+1}$. If $\alpha_i>\alpha_{i+1}$, note that $\nu_i\geq \lambda_i+n\alpha_i-|\mu|$ and $\nu_{i+1}\leq \lambda_{i+1}+n\alpha_{i+1}$, we have
\begin{align*}
\nu_i-\nu_{i+1}&\geq \lambda_i+n\alpha_i-|\mu|-(\lambda_{i+1}+n\alpha_{i+1})\geq n(\alpha_i-\alpha_{i+1})-|\mu|\\
&\geq (n-|\mu|)(\alpha_i-\alpha_{i+1}).
\end{align*}
So \eqref{-partition} holds, and we have proved part (1).

Part (2) follows from part (1) and Proposition \ref{LR-StabResult}, as $((\alpha), (0), (\alpha))$ is LR-stable.
\end{proof}
\noindent Similarly, we have the following result:
\begin{lem}
	\label{LRLemma'}
	Given partitions $\lambda$, $\mu$, $\nu$, $\alpha$, and an positive integer $n\geq |\mu|$, with $|\lambda|=|\mu|+|\nu+n\alpha|$ and $\nu+n\alpha\subset \lambda$, then
	\begin{itemize}
		\item[(1)] $\lambda-(n-|\mu|)\alpha$ is a partition.
		
		\item[(2)] When $n$ is large, we have $c_{\mu, \nu+n\alpha}^{\lambda}=c_{\mu, \nu+(n+1)\alpha}^{\lambda+\alpha}.$
	\end{itemize}
\end{lem}

\begin{rem}\label{LRBound}
	Proposition \ref{LR-StabResult} does not give a lower bound for what large means for $n$ in part (2) of Lemma \ref{LRLamma}. However, it is not hard to see that when $n\geq 2|\mu|$, the connected components of the skew shapes $(\lambda+n\alpha)/\nu$ and $(\lambda+(n+1)\alpha)/(\nu+\alpha)$ are the same except for some horizontal shifts, hence, due to the Littlewood--Richardson rule, $c_{\mu, \nu}^{\lambda+n\alpha}=c_{\mu, \nu+\alpha}^{\lambda+(n+1)\alpha}$. Similarly, for Lemma \ref{LRLemma'} (2), $n\geq 2|\mu|$ is enough to guarantee the stability.
\end{rem}
The next theorem generalizes the result in \cite{LY}.
\begin{thm}\label{KisH}
	A K-stable triple is H-stable.	
\end{thm}
\noindent Let $(\alpha, \beta, \gamma)$ with $\alpha, \beta, \gamma\vdash s>0$ be K-stable. Suppose $\lambda, \mu,$ and $\nu$ are partitions with $\lambda\vdash p$, $\mu\vdash q$, and $\nu\vdash r$ and $\max\{q,r\}\leq p\leq q+r$. Theorem \ref{KisH} states that the sequence $\left\{h_{\mu+n\beta, \nu+n\gamma}^{\lambda+n\alpha}\right\}$ is eventually constant. According to Proposition \ref{formulaheisenberg}, we have
\begin{equation}
\label{H=CGcomplex1}
h_{\mu+n\beta, \nu+n\gamma}^{\lambda+n\alpha}=\sum\limits_{K_n} c_{\xi, \theta }^{\mu+n\beta}\,\, c_{\eta, \rho}^{\nu+n\gamma}\,\, g_{\theta, \eta}^\delta\,\, c_{\xi, \delta}^\tau\,\, c_{\tau, \rho}^{\lambda+n\alpha},
\end{equation}
where 
$$K_n=\{(\xi, \theta, \eta, \rho, \delta, \tau)\mid \theta, \eta, \delta\vdash (q+r-p)+ns,\,\, \xi\in p-r,\,\, \rho\in p-q,\,\, \tau\vdash q+ns\}.$$
Define $f_n: K_n\longmapsto \mathbb{Z}_{\geq 0}$ as follows:
\begin{equation}
\label{definef}
f_n(\xi, \theta, \eta, \rho, \delta, \tau)=c_{\xi, \theta }^{\mu+n\beta}\,\, c_{\eta, \rho}^{\nu+n\gamma}\,\, g_{\theta, \eta}^\delta\,\, c_{\xi, \delta}^\tau\,\, c_{\tau, \rho}^{\lambda+n\alpha}\quad (\text{the summands in}\; \eqref{H=CGcomplex1}).
\end{equation}
Some terms in the sum of \eqref{H=CGcomplex1} vanish. Let us consider only the nonvanishing terms. Let $K_n^0 =K_n \smallsetminus f_n^{-1}(0)$.
To prove Theorem \ref{KisH}, it is enough to prove
\begin{equation}
\label{fn=fn+1}
\sum\limits_{u\in K_n^0} f_n(u)=\sum\limits_{u\in K_{n+1}^0} f_{n+1}(u)
\end{equation}
for $n$ sufficiently large.

We have a natural embedding $\varphi_n: K_n\hookrightarrow K_{n+1}$, 
$$\varphi_n (\xi, \theta, \eta, \rho, \delta, \tau)=(\xi, \theta+\beta, \eta+\gamma, \rho, \delta+\alpha, \tau+\alpha),$$we show that when $n$ is large, $\varphi_n$ induces a bijection between $K_n^0$ and $K_{n+1}^0$ with $f_n=f_{n+1}\circ \varphi_n$. From the definition of K-stability, we know that there exists a positive integer $N$, such that for all $n\geq N$, we have
\begin{equation}\label{g=g+}
g_{\zeta+n\beta, \pi+n\gamma}^{\epsilon+n\alpha}=g_{\zeta+(n+1)\beta, \pi+(n+1)\gamma}^{\epsilon+(n+1)\alpha},
\end{equation}
for all $\epsilon, \zeta, \pi\vdash q+r-p+(2p-q-r)s$.
\begin{lem}\label{LemmaforK}
	When $n\geq N+3p-q-2r$, $\varphi_n|_{K_n^0}$: $K_n^0 \longrightarrow K_{n+1}^0$ is a well-defined bijection. Moreover, $f_n|_{K_n^0}=f_{n+1}\circ \varphi_n|_{K_n^0}.$ 
\end{lem}
\begin{proof}
	Take any $u=(\xi, \theta, \eta, \rho, \delta, \tau)\in K_n^0$. Since $f(u)\neq 0$, we have $c_{\xi, \theta}^{\mu+n\beta}\neq 0$. So $\theta\subset \mu+n\beta$. According to Lemma \ref{LRLamma} and Remark \ref{LRBound}, we know that $\theta-(n-p+r)\beta$ is a partition of $(q+r-p)+(p-r)s$ and $c_{\xi, \theta}^{\mu+n\beta}=c_{\xi, \theta+\beta}^{\mu+(n+1)\beta}$.
	Similarly, we can show that
	$\eta-(n-p+q)\gamma$, $\tau-(n-p+q)\alpha$, and $\delta-(n-2p+q+r)\alpha$ are partitions, and
	$$c_{\eta, \rho}^{\nu+n\gamma}=c_{\eta, \rho+\gamma}^{\nu+(n+1)\gamma},\quad c_{\tau, \rho}^{\lambda+n\alpha}= c_{\tau+\alpha, \rho}^{\lambda+(n+1)\alpha},\quad c_{\xi, \delta}^{\tau}=c_{\xi, \delta+\alpha}^{\tau+\alpha}.$$
	Since $\delta$, $\theta$, and $\eta$ can be written as 
	$$\delta=\delta'+(n-2p+q+r)\alpha,\; \theta=\theta'+(n-2p+q+r)\beta,\; \eta=\eta'+(n-2p+q+r)\gamma$$
	for some partitions $\delta', \theta',$ and $\eta'$ of $(q+r-p)+(2p-q-r)s$. From \eqref{g=g+}, we have
	$$g_{\eta, \rho}^\delta=g_{\eta+\beta, \rho+\gamma}^{\delta+\alpha}.$$
	Hence, $f_{n+1}(\varphi_n(u))=f_n(u) (\neq 0)$. So $\varphi_n|_{K_n^0}$ is a well-defined embedding from $K_n^0$ into $K_{n+1}^0$. To construct the inverse map, we consider $\psi_{n+1} (\xi, \theta, \eta, \rho, \delta, \tau)=(\xi, \theta-\beta, \eta-\gamma, \rho, \delta-\alpha, \tau-\alpha)$. Nearly the same arguments show that the inverse map induces an injection from $K_{n+1}^0$ to $K_n^0$. So $\varphi_n|_{K_n^0}$ is a bijection.
\end{proof}

\begin{proof}[Proof of Theorem \ref{KisH}]
	Applying Lemma \ref{LemmaforK}, we prove \eqref{fn=fn+1}, hence Theorem \ref{KisH}. 
\end{proof}

Theorem \ref{KisH} shows that some Heisenberg coefficients in low degree components stabilize. Our next result gives a stability result for the relatively high degree components.

\begin{thm}
	\label{LRisH}
	LR-stable triples are H-stable.
\end{thm}
The idea of the proof is the same (without using stability of Kronecker coefficients) as the proof for Theorem \ref{KisH}. Given an LR-stable triple $(\alpha, \beta, \gamma)$ with $\alpha\vdash a+b$, $\beta\vdash a$, and $\gamma\vdash b$, and partitions $\lambda\vdash p$, $\mu\vdash q$, and $\nu\vdash r$ with $\max\{q,r\}\leq p\leq q+r$. We define $f'_n: LR_n\longmapsto \mathbb{Z}_{\geq 0}$ as follows:
\begin{equation}
\label{definef'}
f'_n(\xi, \theta, \eta, \rho, \delta, \tau)=c_{\xi, \theta }^{\mu+n\beta}\,\, c_{\eta, \rho}^{\nu+n\gamma}\,\, g_{\theta, \eta}^\delta\,\, c_{\xi, \delta}^\tau\,\, c_{\tau, \rho}^{\lambda+n\alpha}.
\end{equation}
where 
$$LR_n=\{(\xi, \theta, \eta, \rho, \delta, \tau)\mid \theta, \eta, \delta\vdash (q+r-p),\,\, \xi\vdash p-r+na,\,\, \rho\vdash p-q+nb,\,\, \tau\vdash q+na\}.$$ 
Applying Proposition \ref{H=CGcomplex1}, Theorem \ref{LRisH} states that
\begin{equation}
\label{f'_n=f'n+1}
\sum\limits_{u\in LR_n^0} f'_n(u)=\sum\limits_{u\in LR_{n+1}^0} f'_{n+1}(u)
\end{equation}
for all large $n$, where $LR_n^0=LR_n\smallsetminus {f'}_n^{-1}(0)$.

\begin{proof}[Proof of Theorem \ref{LRisH}]
	Consider the map $\phi_n: LR_n\hookrightarrow LR_{n+1}$,
	$$\phi_n(\xi, \theta, \eta, \rho, \delta, \tau)=(\xi+\beta, \theta, \eta, \rho+\gamma, \delta, \tau+\beta).$$
	Using Lemma \ref{LRLamma}, Lemma \ref{LRLemma'}, and the same idea in the proof of Lemma \ref{LemmaforK}, it follows that $f'_n=f'_{n+1}\circ \phi_n$ on $LR_n^0$ when $n$ is large, and it is not hard to see that $\phi_n$ is a bijection between $LR_n^0$ and $LR_{n+1}^0$. So \eqref{f'_n=f'n+1} is true, and hence we prove the theorem.
\end{proof}

Proposition \ref{K-stable} and Proposition \ref{LR-StabResult} (3) have a similar form. A natural question is whether the necessary and sufficient condition for being an H-stable triple has the same form. The answer is yes, and Pelletier \cite[Theorem 3.6]{Pelletier2} proved the following direction.

\begin{prop}
	\label{Hsuffcient}
	An H-triple $(\alpha, \beta, \gamma)$ is H-stable if $h_{n\beta, n\gamma}^{n\alpha}=1$ for all $n>0$.
\end{prop}

\begin{rem}
	By Propositions \ref{K-stable} and \ref{LR-StabResult}, Proposition \ref{Hsuffcient} also shows that K-stable triples and LR-stable triples are H-stable.
\end{rem}
 
We prove the reverse direction and complete the characterization of H-stability using monotonicity of the Heisenberg coefficients. This is deduced from the monotonicity of Littlewood--Richardson coefficients and Kronecker coefficients. We start with the monotonicity of Kronecker coefficients. Stembridge \cite{JT} proved the following for Kronecker coefficients.

\begin{prop}\label{KMon}
	Let $(\alpha, \beta, \gamma)$ be a K-triple. Then
	\begin{itemize}
		\item[(1)] the sequence $\left\{g_{\mu+n\beta, \nu+n\gamma}^{\lambda+n\alpha}\right\}$ is weakly increasing for any partitions $\lambda$, $\mu$, and $\nu$ with the same size.
		\item[(2)]  if $g^\alpha_{\beta, \gamma}\geq 2$, then $g_{n\beta, n\gamma}^{n\alpha}\geq n+1$.
	\end{itemize} 
\end{prop}
\noindent Using the hive model of Littlewood-Richardson coefficients (see \cite{KnutsonTao, CombinLR}), we prove an analogous result for Littlewood-Richardson coefficients.

\begin{prop}\label{LRMon}
		Let $(\alpha, \beta, \gamma)$ be an LR-triple. Then
	\begin{itemize}
		\item[(1)] the sequence $\left\{c_{\mu+n\beta, \nu+n\gamma}^{\lambda+n\alpha}\right\}$ is weakly increasing for any partitions $\lambda$, $\mu$, and $\nu$ with $|\lambda|=|\mu|+|\nu|$.
		\item[(2)]  if $c^\alpha_{\beta, \gamma}\geq 2$, then $c_{n\beta, n\gamma}^{n\alpha}\geq n+1$.
	\end{itemize} 
\end{prop}
\begin{proof}
	We follow the notation used for hives in \cite[Section 4]{CombinLR}. Let $k$ be a positive integer larger than the lengths of $\lambda$, $\mu$, $\nu$, $\alpha$, $\beta$, and $\gamma$. We define (coordinatewise) addition and scalar multiplication on hives (as what we do for vectors and matrices).
	
	For (1), it suffices to show $c_{\mu,\nu}^\lambda\leq c_{\mu+\beta, \nu+\gamma}^{\lambda+\alpha}$. Since $c_{\beta, \gamma}^\alpha\geq 1$, there exists a hive $\Delta \in H_k(\alpha, \beta, \gamma)$. Then the map: $\iota: H_k(\lambda, \mu, \nu)\hookrightarrow H_k(\lambda+\alpha, \mu+\beta, \nu+\gamma)$
	$$\iota(\Theta)=\Theta+\Delta$$
	where $\Theta\in H_k(\lambda, \mu, \nu)$, gives a well-defined injection. So (1) is proved.
	
	For (2), we have two different hives $\Delta_1$ and $\Delta_2$ in $H_k(\alpha, \beta, \gamma)$ as $c_{\beta, \gamma}^\alpha\geq 2$. Then $i\Delta_1+(n-i)\Delta_2$ ($0\leq i\leq n$) give $n+1$ different hives in $H_k(n\alpha, n\beta, n\gamma)$, so $c_{n\beta, n\gamma}^{n\alpha}\geq n+1$.
\end{proof}
Propositions \ref{formulaheisenberg}, \ref{KMon}, and \ref{LRMon} together imply the following:
\begin{prop}\label{HMon}
		Let $(\alpha, \beta, \gamma)$ be a H-triple. Then
		\begin{itemize}
			\item[(1)] the sequence $\left\{h_{\mu+n\beta, \nu+n\gamma}^{\lambda+n\alpha}\right\}$ is weakly increasing for any partitions $\lambda$, $\mu$, and $\nu$ with $\max\{|\mu|, |\nu|\}\leq |\lambda|\leq |\mu|+|\nu|$.
			\item[(2)]  if $h^\alpha_{\beta, \gamma}\geq 2$, then $h_{n\beta, n\gamma}^{n\alpha}\geq n+1$.
		\end{itemize} 
\end{prop}
\begin{proof}
For (1), it is enough to show $h_{\mu, \nu}^\lambda\leq h_{\mu+\beta, \nu+\gamma}^{\lambda+\alpha}$. Since $h_{\beta, \gamma}^\alpha>0$, by Formula \eqref{h=cg}, there exists a sextuple $(\xi, \theta, \eta, \rho, \delta, \tau)$ of partitions with appropriate sizes such that $c_{\xi, \theta }^{\beta}\,\, c_{\eta, \rho}^{\gamma}\,\, g_{\theta, \eta}^\delta\,\, c_{\xi, \delta}^\tau\,\, c_{\tau, \rho}^{\alpha}>0$. The triples appearing in the coefficients on the left hand side are LR-triples or K-triples. As in the proof of Theorems \ref{KisH} and \ref{LRisH}, applying Proposition \ref{formulaheisenberg}, we write 
$$h_{\mu, \nu}^\lambda=\sum\limits_{u\in\Lambda} f_u \quad\text{and}\quad h_{\mu+\beta, \nu+\gamma}^{\lambda+\alpha}=\sum\limits_{u'\in\Lambda'} f'_{u'},$$ where $\Lambda$ and $\Lambda'$ are sets of sextuples of partitions with appropriate sizes, $f_u$ and $f'_{u'}$ are the summands given by the sextuples $u$ and $u'$. We view the sextuples as vectors whose coordinates are partitions, so we may define addition and scalar multiplication for them. The map $u\longrightarrow u+(\xi, \theta, \eta, \rho, \delta, \tau)=:u'$ embeds $\Lambda$ into $\Lambda'$. From Proposition \ref{KMon} and Proposition \ref{LRMon}, we know that $f_u\leq f'_{u'}$, so (1) is proved. 

For (2), if $h^\alpha_{\beta, \gamma}\geq 2$, then there are two possibilities.

\noindent Case 1. There exists a sextuple $(\xi, \theta, \eta, \rho, \delta, \tau)$ of partitions with appropriate sizes such that $c_{\xi, \theta }^{\beta}\,\, c_{\eta, \rho}^{\gamma}\,\, g_{\theta, \eta}^\delta\,\, c_{\xi, \delta}^\tau\,\, c_{\tau, \rho}^{\alpha}\geq 2.$ So all the five coefficients on the left hand side are positive and at least one of them is at least 2. From Propositions \ref{KMon} and \ref{LRMon}, we have 
$$h_{n\beta, n\gamma}^{n\alpha}\geq c_{n\xi, n\theta }^{n\beta}\,\, c_{n\eta, n\rho}^{n\gamma}\,\, g_{n\theta, n\eta}^{n\delta}\,\, c_{n\xi, n\delta}^{n\tau}\,\, c_{n\tau, n\rho}^{n\alpha}\geq n+1.$$

\noindent Case 2. Two distinct sextuples $u=(\xi, \theta, \eta, \rho, \delta, \tau)$ and $u'=(\xi', \theta', \eta', \rho', \delta', \tau')$ give positive summands for $h_{\beta, \gamma}^{\alpha}$. Then $iu+(n-i)u'$ ($1\leq i\leq n$) gives $n+1$ different sextuples, and due to Propositions \ref{KMon} and \ref{LRMon}, they all give positive summands for $h_{n\beta, n\gamma}^{n\alpha}$, so $h_{n\beta, n\gamma}^{n\alpha}\geq n+1$. Hence, we prove the proposition.
\end{proof}
\noindent Combining Proposition \ref{Hsuffcient} and \ref{HMon}, we achieve the main theorem of this paper.
\begin{thm}
	An H-triple $(\alpha, \beta, \gamma)$ is H-stable if and only if $h_{n\beta, n\gamma}^{n\alpha}=1$ for all $n>0$.
\end{thm}

\section{Additive Matrices}
\label{SecAdditive}
Manivel \cite{ManKronCone} and Vallejo \cite{VallejoStab} used additive matrices to produce examples of K-stable triples. We first recall some definitions and results concerning additive matrices, then we give an analogous result for H-stable triples.

For positive integer $n$, let $[n]:=\{1,2,\ldots, n\}$. Given a matrix $A$, we arrange its entries in weakly decreasing order. The result sequence is called the $\pi$-sequence of $A$, and denoted by $\pi(A)$. Let $\mathcal{M}(\alpha, \beta)$ denote the set of matrices with nonnegative integer entries, row-sum vector $\alpha$ and column-sum vector $\beta$, and $\mathcal{M}(\alpha, \beta)_\gamma$ be the subset of $\mathcal{M}(\alpha, \beta)$ whose elements have $\pi$-sequence $\gamma$.

\begin{defn}
A $p\times q$ matrix $A=(a_{i,j})$ with nonnegative integer entries is called \textit{K-additive} if there exist real numbers $x_1, x_2, \ldots, x_p, y_1, y_2, \ldots, y_q$, such that
$$a_{i,j}>a_{k,l}\Longrightarrow x_i+y_j>x_k+y_l$$
for all $i,k\in [p]$ and $j,l\in[q].$
\end{defn}
\begin{prop}[\cite{VallejoStab} Theorem 1.1]\label{AddisK}
	Let $\alpha, \beta$, and $\gamma$ be partitions of the same size. If there is a matrix $A\in \mathcal{M}(\beta, \gamma)_\alpha$ which is K-additive, then $(\alpha, \beta, \gamma)$ is K-stable.
\end{prop}
\noindent Moreover, Manivel \cite[Section 5.3]{ManKronCone} showed that each K-additive matrix defines a regular face of the corresponding Kronecker polyhedron, of minimal dimension. 

For any (weak) composition $\alpha=(\alpha_1, \alpha_2, \ldots, \alpha_r)$ of $n$ (written as $\alpha\vDash n$), let $\mathfrak{h}_\alpha$ be the induced representation from the trivial representation of the Young subgroup $S_\alpha=S_{\alpha_1}\times S_{\alpha_2}\times \dotsb \times S_{\alpha_r}$ to $S_n$. \noindent For partitions $\lambda$ and $\mu$ with the same size, the multiplicity of $V_\lambda$ in the irreducible decomposition of $\mathfrak{h}_\mu$ is $K_{\lambda, \mu}$, the Kostka number. This counts the number of semistandard Young tableaux of shape $\lambda$ and content $\mu$. It is well-known that $K_{\lambda, \mu}>0$ if and only if $\lambda\succcurlyeq \mu$. In particular, $K_{\lambda, \lambda}=1$. So
\begin{equation}\label{htos}
\mathfrak{h}_\mu=\bigoplus\limits_{\lambda\succcurlyeq \mu} K_{\lambda, \mu}V_\lambda=V_\mu\oplus\left(\bigoplus\limits_{\lambda\succ \mu} K_{\lambda, \mu}V_\lambda\right).
\end{equation}
For a (weak) composition $\alpha$ of $|\lambda|$, $\mathfrak{h}_\alpha$ is isomorphic to $\mathfrak{h}_{\pi(\alpha)}$ as representations of $S_{|\lambda|}$, so we define $K_{\lambda, \alpha}=K_{\lambda, \pi(\alpha)}$. One of the most important steps in Vallejo's proof of Proposition \ref{AddisK} is the following formula, which computes the Kronecker product of two $\mathfrak{h}$'s.
\begin{prop}\label{Kronofh}
	Let $\beta$ and $\gamma$ be (weak) compositions of $n$, then the Kronecker product of $\mathfrak{h}_\beta$ and $\mathfrak{h}_\gamma$ is
	$$\text{Res}_{S_n}^{S_n\times S_n}(\mathfrak{h}_\beta\otimes\mathfrak{h}_\gamma)=\bigoplus\limits_{A\in \mathcal{M}(\beta, \gamma)} \mathfrak{h}_{\pi(A)}=\bigoplus\limits_{\alpha\vdash n}\bigoplus\limits_{A\in \mathcal{M}(\beta, \gamma)_{\alpha}} \mathfrak{h}_{\alpha}.$$
\end{prop}
\noindent Aguiar et al. ~\cite{M-2015} provided a similar formula for the Heisenberg product. To express the formula, we introduce some notation. Given two finite sequences of real numbers $\alpha$, $\beta$, and $\gamma$. Let $\mathcal{F}(\alpha, \beta)$ be the set of matrices with real entries, zero at the top left corner, row-sum vector (ignoring the first row) $\alpha$ and column-sum vector (ignoring the first column) $\beta$. We denote by $\mathcal{H}(\alpha, \beta)$ the set of matrices in $\mathcal{F}(\alpha, \beta)$ with integer entries, and $\mathcal{H}(\alpha, \beta)_{\gamma}$ the subset of $\mathcal{H}(\alpha, \beta)$ whose elements have $\pi$-sequence $\gamma$.

\begin{exmp}\label{Hmatrixex}
	The following matrix is in $\mathcal{H}((18, 10), (12, 18, 3))_{(7,6,5,5,4,4,3,2,2,1)}$
	$$\left(\begin{array}{cccc}
		0 & 4 & 6 & 1\\
		4 & 5 & 7 & 2\\
		2 & 3 & 5 & 0
	\end{array}\right)_.$$
\end{exmp}

\begin{prop}[\cite{M-2015} Theorem 3.1]\label{Heisenofh}
	Let $\beta$ and $\gamma$ be two (weak) compositions, then the Heisenberg product of $\mathfrak{h}_\beta$ and $\mathfrak{h}_\gamma$ is
	$$\mathfrak{h}_\beta\#\mathfrak{h}_\gamma=\bigoplus\limits_{A\in \mathcal{H}(\beta, \gamma)} \mathfrak{h}_{\pi(A)}.$$
\end{prop}
\noindent We introduce H-additive matrices and use Proposition \ref{Heisenofh} to show that each H-additive matrix gives an H-stable triple.
\begin{defn}\label{HAdd}
	A $(p+1)\times (q+1)$ matrix $A=(a_{i,j})$ with nonnegative integer entries and $a_{1,1}=0$ is called\textit{ H-additive} if there exist real numbers $x_1=0, x_2, \ldots, x_{p+1}, y_1=0, y_2, \ldots, y_{q+1}$, such that
	$$a_{i,j}>a_{k,l}\Longrightarrow x_i+y_j>x_k+y_l$$
	for all  $(i,j), (k,l)\in [p+1]\times [q+1]\smallsetminus \{(1,1)\}$.
\end{defn}
With this definition, the matrix in Example \ref{Hmatrixex} is H-additive (consider setting $x_0=y_0=0, x_1=1, x_2=-1, y_1=1, y_2=3, y_3=-2$).
\begin{thm}\label{AddisH}
		Let $\alpha, \beta$, and $\gamma$ be partitions with $\max\{|\beta|, |\gamma|\}\leq |\alpha|\leq |\beta|+|\gamma|$. If there is a matrix $A\in \mathcal{H}(\beta, \gamma)_\alpha$ which is H-additive, then $(\alpha, \beta, \gamma)$ is H-stable.
\end{thm}
\begin{rem}\label{Hgeneralize}
	Theorem \ref{AddisH} is equivalent to Proposition \ref{AddisK} if $|\alpha|=|\beta|=|\gamma|$. The only LR-stable triples it can produce are in the form $(\beta\cup \gamma, \beta, \gamma)$, where $\beta\cup \gamma$ is the partition whose parts are those of $\beta$ and $\gamma$, arranged in decreasing order. It is not hard to see $c_{\beta, \gamma}^{\beta \cup \gamma}=1$.
\end{rem}
The proof for Theorem \ref{AddisH} is similar to Onn and Vallejo's proof \cite{Minimalmatrices, VallejoStab} for Proposition \ref{AddisK} with some changes, as we are looking at slightly different matrices. Consequently, we only give a sketch.

We first recall some basic notions, many introduced in \cite{Minimalmatrices,VallejoStab}, which will be used in our proof. We move away from integers for a while and work with real numbers. For a vector $\textbf{a}=(a_1, a_2, \ldots, a_m)\in \mathbb{R}^m$, we denote by $\pi(\textbf{a})$ the vector formed by the entries of $\textbf{a}$ arranged in weakly decreasing order. We say that $\textbf{a}$ is dominated by $\textbf{b}$ (both are vectors in $\mathbb{R}^m$), written as $\textbf{a}\preccurlyeq \textbf{b}$, if
$$\sum\limits_{i=1}^m a_i=\sum\limits_{i=1}^m b_i \quad \text{and}\quad \sum\limits_{i=1}^k \pi(\textbf{a})_i\leq\sum\limits_{i=1}^k \pi(\textbf{b})_i, \hskip 2mm \text{for all}\;\; k\in[m].$$
If $\textbf{a}\preccurlyeq \textbf{b}$ and $\pi(\textbf{a})\neq \pi(\textbf{b})$, then we write $\textbf{a}\prec \textbf{b}$. In particular, when $\textbf{a}$ and $\textbf{b}$ are partitions of some integer $n$, then $\preccurlyeq$ coincides with the dominance order for partitions.

For a permutation $\rho\in S_m$, we set $\textbf{a}_\rho:=(a_{\rho(1)}, a_{\rho(2)}, \ldots, a_{\rho(m)})$. The \textit{permutohedron} determined by $\textbf{a}$ is the convex hull of the vectors of the form $a_{\rho}$:
$$P(\textbf{a})=\text{conv}\{\textbf{a}_\rho\,|\,\rho\in S_m\}.$$
\begin{prop}[Rado \cite{Permutohedron}]\label{Permleq}\qquad$P(\textbf{a})=\{\textbf{x}\in \mathbb{R}^m\,|\,\textbf{x}\preccurlyeq \textbf{a}\}.$
\end{prop}
Suppose $\alpha$ and $\beta$ are two finite sequences of real numbers whose lengths are $p$ and $q$ respectively. We consider the linear map $\Phi: \mathcal{F}(\alpha, \beta)\longrightarrow \mathbb{R}^{pq+p+q}$,
$$\Phi(A)=(a_{1,2}, a_{1,3}, \ldots, a_{1,q+1}, a_{2,1}, a_{2,2}, \ldots, a_{2,q+1}, \ldots, a_{p+1,1}, a_{p+1,2}, a_{p+1,q+1})$$
where $A=(a_{i,j})\in \mathcal{F}(\alpha, \beta)$. A matrix $A\in \mathcal{F}(\alpha, \beta)$ is real-minimal if there is no other matrix $B\in \mathcal{F}(\alpha, \beta)$ such that $\pi(B)\prec \pi(A)$. Real-minimality has the following equivalent interpretations. 
\begin{prop}\label{mineqv}
	Let $A\in \mathcal{F}(\alpha, \beta)$. The following are equivalent:
	\begin{itemize}
		\item[(1)] $A$ is real-minimal.
		
		\item[(2)] $P(\Phi(A))\cap \Phi(\mathcal{F}(\alpha, \beta))=\{\Phi(A)\}$.
		
		\item[(3)] there exists a hyperplane $H\subset \mathbb{R}^{pq+p+q}$ containing $\Phi(\mathcal{F}(\alpha, \beta))$ such that 
		$$P(\Phi(A))\cap H=\{\Phi(A)\}.$$
	
	\end{itemize}
\end{prop}
\noindent See \cite[Section 5]{Minimalmatrices} for the proof of this Proposition. Although the matrices we are working with are different, the proof still applies.

\begin{prop}\label{minortho}
	Let $A\in \mathcal{F}(\alpha, \beta)$, $\textbf{a}=\Phi(A)$. Then $A$ is real-minimal if and only if there is some vector $\textbf{n}\in \mathbb{R}^{pq+p+q}$ such that
	\begin{itemize}
		\item[(1)] $\textbf{n}$ is orthogonal to $\Phi(\mathcal{F}(\alpha, \beta))$.
		\item[(2)] For each transposition $\sigma=(s\; s{+}1)\in S_{pq+p+q}$ such that $a_s\neq a_{s+1}$, one has $\langle\, \textbf{n}\,, \sigma \textbf{a} -\textbf{a}\, \rangle > 0$.
	\end{itemize}
\end{prop}
\begin{rem}
The second condition is equivalent to $\langle\, \textbf{n}\,, \textbf{x} -\textbf{a}\, \rangle > 0$ for all $\textbf{x}\in P(\textbf{a}), \textbf{x}\neq \textbf{a}$. 
\end{rem}
\noindent Again, one can use the proof in \cite[Proposition 6.1]{Minimalmatrices} to prove this. The definition of H-additivity can be extended naturally to matrices with real entries, and we next show that real-minimality is equivalent to H-additivity for real matrices.
\begin{thm}\label{min=add}
	Let $A\in  \mathcal{F}(\alpha, \beta)$. Then $A$ is real-minimal if and only if $A$ is H-additive.
\end{thm}
\noindent Following Onn and Vallejo \cite{Minimalmatrices}, we first construct a matrix. Let $M=(m_{i,j})$ be a $(p+q)\times (pq+p+q)$ matrix with
$$m_{i,j}=\left\{{\begin{array}{l}
1,\quad \text{if}\;\; 1\leq i\leq p \;\;\text{and}\;\; i(q+1)\leq j \leq q+i(q+1);\\
1,\quad \text{if}\;\; p+1\leq i\leq p+q\;\; \text{and}\;\; j=s-p+k(q+1),\; \text{for some}\;\; 0\leq k\leq p;\\
0,\quad \text{otherwise.}
\end{array}}\right.$$
For example, if $p=2$ and $q=3$, then
$$M=\left(\begin{array}{ccccccccccc}
0&0&0&1&1&1&1&0&0&0&0\\
0&0&0&0&0&0&0&1&1&1&1\\
1&0&0&0&1&0&0&0&1&0&0\\
0&1&0&0&0&1&0&0&0&1&0\\
0&0&1&0&0&0&1&0&0&0&1
\end{array}
\right).$$

Let $\textbf{r}_1, \textbf{r}_2, \ldots, \textbf{r}_{p+q}$ be the rows of $M$. From the definition of $M$, it is obvious that these $p+q$ vectors are linearly independent. The set $\Phi(\mathcal{F}(\alpha, \beta))$ is exactly the set of (transpose of) solutions of the following matrix equation:
$$M\textbf{x}=\textbf{y}$$
where $\textbf{y}=(\alpha_1,\ldots, \alpha_p, \beta_1, \ldots, \beta_q)'.$ For a vector $\textbf{z}=(x_2, \ldots x_{p+1}, y_2, \ldots, y_{q+1})$, we have
$$\textbf{z}M=\Phi((x_i+y_j)_{(i,j)\in [p+1]\times [q+1]}),$$
where we set $x_1=y_1=0$.
\begin{proof}[Proof of Theorem \ref{min=add}]
	Suppose $A\in  \mathcal{F}(\alpha, \beta)$ is real-minimal, then there exists a vector $\textbf{n}\in \mathbb{R}^{pq+p+q}$ satisfying the two conditions in Proposition \ref{minortho}. Since $\textbf{n}$ is orthogonal to $\Phi(\mathcal{F}(\alpha, \beta))$, $\textbf{n}$ must be in the row space of $M$. So there are unique numbers $x_2, \ldots, x_{p+1}, y_2, \ldots, y_{q+1}$ such that
	$$-\textbf{n}=x_2 \textbf{r}_1+\cdots+x_{p+1} \textbf{r}_p+y_2\textbf{r}_{p+1}+\cdots+y_{q+1}\textbf{r}_{p+q}.$$
	Let $\textbf{z}=(x_2, \ldots x_{p+1}, y_1, \ldots, y_{q+1})$, then $-\textbf{n}=\textbf{z}M=-\Phi((x_i+y_j))$. Following the arguments in \cite[Theorem 6.2]{Minimalmatrices} proves this theorem.
\end{proof}
\noindent From Proposition \ref{Permleq}, \ref{mineqv}, and Theorem \ref{min=add}, we have
\begin{cor}\label{intersectionunique}
	Let $A\in  \mathcal{F}(\alpha, \beta)$. Then $A$ is H-additive if and only if
	$$P(\pi(A))\cap \Phi(\mathcal{F}(\alpha, \beta))=\{\Phi(A)\}.$$
\end{cor}
\section{Proof of Theorem \ref{AddisH}}

Vallejo showed that the Kronecker coefficient indexed by the K-triple produced by a K-additive matrix is 1.
\begin{lem}[\cite{PlanePartChar} Corollary 4.2]\label{add=1}
	Let $A\in \mathcal{M}(\beta, \gamma)_{\alpha}$ be K-additive where $\alpha$, $\beta$, and $\gamma$ are partitions with the same size, then $g_{\beta, \gamma}^{\alpha}=1$.
\end{lem}
\noindent The same is true for Heisenberg coefficients and H-additive matrices.
\begin{lem}\label{Hadd=1}
	Let $A\in \mathcal{H}(\beta, \gamma)_{\alpha}$ be H-additive, where $\alpha$, $\beta$, and $\gamma$ are partitions with $\max\{|\beta|, |\gamma| \}\leq |\lambda|\leq |\beta|+|\gamma|$, then $h_{\beta, \gamma}^{\alpha}=1$.
\end{lem}
\begin{proof}
We first show that $h_{\beta, \gamma}^{\alpha}\geq 1$. Suppose $\beta=(\beta_1, \beta_2, \ldots, \beta_p)$, $\gamma=(\gamma_1, \gamma_2, \ldots, \gamma_q)$. Since $A=(a_{ij})$ is additive, then there exists real numbers $x_i$'s and $y_j$'s ($i\in [p+1]$ and $ j\in[q+1]$) satisfy the condition in Definition \ref{HAdd}. After permuting rows and columns if necessary, we may assume $x_2\geq x_3\geq \cdots \geq x_{p+1}$ and $y_2\geq y_3\geq \cdots \geq y_{q+1}$. By H-additivity, this assumption implies that  $a_{i,j}\geq a_{i,j+1}$ for all $1\leq i\leq p+1, 2\leq j\leq q$, and $a_{i,j}\geq a_{i+1,j}$ for all $2\leq i\leq p, 1\leq j\leq q+1$. Set $\beta^{(1)}=(a_{2,1}, a_{3,1}, \ldots, a_{p+1, 1})$ , $\gamma^{(1)}=(a_{1,2}, a_{1,3}, \ldots, a_{1, q+1})$, $\beta^{(2)}=\beta-\beta^{(1)}$, and $\gamma^{(2)}=\gamma-\gamma^{(1)}$, then these four are all partitions and, by the Littlewood--Richardson rule, we have
\begin{equation}
\label{TwoLR=1(1)}
c_{\beta^{(1)}, \beta^{(2)}}^\beta=c_{\gamma^{(1)}, \gamma^{(2)}}^{\gamma}=1.
\end{equation}
Let $A^{(1)}$ be the the submatrix of $A$ obtained by removing the first row, and $A^{(2)}$ be the submatrix of $A^{(1)}$ obtained by removing the first column. We set $\alpha^{(1)}=\pi(A^{(1)})$ and $\alpha^{(2)}=\pi(A^{(2)})$. From Remark \ref{Hgeneralize}, we have
\begin{equation}
\label{TwoLR=1(2)}
c_{\gamma^{(1)}, \alpha^{(1)}}^{\alpha}=c_{\beta^{(1)}, \alpha^{(2)}}^{\alpha^{(1)}}=1,
\end{equation}
as $\alpha=\gamma^{(1)}\cup \alpha^{(1)}$ and $\alpha^{(1)}=\beta^{(1)}\cup \alpha^{(2)}$. Note that $A^{(2)}\in \mathcal{M}({\beta^{(2)}, \gamma^{(2)}})_{\alpha^{(2)}}$ is additive, so, due to Lemma \ref{add=1}, we have 
\begin{equation}
\label{Kron=1}
g_{\beta^{(2)}, \gamma^{(2)}}^{\alpha^{(2)}}=1.
\end{equation}
Using Proposition \ref{formulaheisenberg} and Equations \eqref{TwoLR=1(1)}, \eqref{TwoLR=1(2)}, and \eqref{Kron=1}, we have $h_{\beta, \gamma}^\alpha\geq 1$.

On the other hand, using Equation \eqref{htos} and properties of Kostka numbers, we have
\begin{equation}\label{hupper}
\begin{split}
h_{\beta, \gamma}^\alpha&=\langle\, V_{\beta}\# V_{\gamma}\,,\, V_{\alpha}\,\rangle\leq \langle\, \mathfrak{h}_{\beta}\# \mathfrak{h}_{\gamma}\,,\, V_{\alpha}\,\rangle=\langle\,\bigoplus\limits_{A\in \mathcal{H}(\beta, \gamma)} \mathfrak{h}_{\pi(A)} \,,\, V_{\alpha}\,\rangle\\
&=\langle\,\bigoplus_{\delta} |\mathcal{H}(\beta, \gamma)_\delta|\, \mathfrak{h}_{\delta} \,,\, V_{\alpha}\,\rangle=\langle\,\bigoplus_{\delta}\bigoplus\limits_{\epsilon\succcurlyeq \delta} |\mathcal{H}(\beta, \gamma)_\delta|K_{\epsilon, \delta}\, V_\epsilon \,,\, V_{\alpha}\,\rangle\\
&=\sum\limits_{\delta\preceq \alpha} |\mathcal{H}(\beta, \gamma)_\delta|K_{\alpha, \delta}
\end{split}
\end{equation}
Since $A\in \mathcal{H}(\beta, \gamma)_\alpha$ is H-additive, according to Corollary \ref{intersectionunique}, we have
$$P(\alpha)\cap \Phi(\mathcal{H}(\beta, \gamma))=\{\Phi(A)\}.$$
Hence, it follows that $|\mathcal{H}(\beta, \gamma)_\delta|=0$ for all $\delta\prec \alpha$, and $|\mathcal{H}(\beta, \gamma)_\alpha|=1$. Equation \eqref{hupper} shows that $h_{\beta, \gamma}^\alpha\leq 1$, proving the lemma.
\end{proof}
\begin{proof}[Proof of Theorem \ref{AddisH}]
If a matrix $A$ is H-additive,then $nA$ is H-additive. Consequently, by Lemma \ref{Hadd=1}, $h_{n\beta, n\gamma}^{n\alpha}=1$ for all $n>0$. By Proposition \ref{Hsuffcient}, $(\alpha, \beta, \gamma)$ is H-stable.
\end{proof}
\begin{rem}
	One may prove Theorem \ref{AddisH} without using Proposition \ref{Hsuffcient}. See \cite[Section 5]{VallejoStab}, the proof is very similar. As in \cite[Theorem 7.1]{Minimalmatrices}, given a rational matrix $A$ with zero at the top left corner, it can be decided in polynomial time whether $A$ is H-additive.
\end{rem}

\bibliographystyle{amsplain}
\bibliography{GeneralizedStab}

\end{document}